\documentclass[12pt]{article}
\usepackage[utf8]{inputenc}
\usepackage{amsmath}
\usepackage[english]{babel}
\usepackage{url}

\usepackage[a4paper,top=3cm,bottom=2cm,left=3cm,right=3cm,marginparwidth=1.75cm]{geometry}

\usepackage{amsmath}
\usepackage{graphicx}
\usepackage[colorinlistoftodos]{todonotes}
\usepackage[colorlinks=true, allcolors=blue]{hyperref}
\usepackage{amssymb}
\usepackage{amssymb}
\usepackage{amsthm}
\usepackage{authblk}
\usepackage{amsmath}
\usepackage{amsfonts}
\usepackage{amssymb}
\usepackage{amscd}
\usepackage{latexsym}
\usepackage{breqn}
\usepackage{cite}
\usepackage{mathtools}

\theoremstyle{plain}

\theoremstyle{definition}

\theoremstyle{remark}

\usepackage{subcaption}
\usepackage{cite}
\usepackage{comment}
\usepackage{mathtools}

\makeatletter
\newenvironment{leftalign*}[1][\parindent]{\setlength\hangindent{#1}\start@align\tw@\st@rredtrue\m@ne}{\endalign}

\makeatother

\usepackage{tikz,amsmath,amsfonts}
\usetikzlibrary{calc,intersections}
\usetikzlibrary{decorations.pathreplacing}
\usepackage{ifthen}
\renewcommand{\phi}{\varphi}

\newcommand{\ran}{\operatorname{ran}}

\renewcommand{\hat}{\widehat}

\let\oldenumerate=\enumerate
	\def\enumerate{
	\oldenumerate
	\setlength{\itemsep}{5pt}
	}
\let\olditemize=\itemize
	\def\itemize{
	\olditemize
	\setlength{\itemsep}{5pt}
	}

\newtheorem{Theorem}{Theorem}[section]
\newtheorem{Lemma}[Theorem]{Lemma}
\newtheorem{Proposition}[Theorem]{Proposition}
\newtheorem{Corollary}[Theorem]{Corollary}
\newtheorem{Definition}[Theorem]{Definition}
\newtheorem{Question}[Theorem]{Question}
\newtheorem{Example}[Theorem]{Example}
\newtheorem{Remark}[Theorem]{Remark}

\allowdisplaybreaks

\numberwithin{equation}{section}
\title{Bounded Toeplitz Products on the Hardy Space}
\author{Ryan O'Loughlin}
\date{}

\begin{document}

\maketitle
\begin{abstract}
A Toeplitz operator on the Hardy space of the unit circle is bounded if and only if its symbol is bounded. For two Toeplitz operators, there are no known function-theoretic conditions for their symbols, which are equivalent to the product of the Toeplitz operators being bounded. In this paper, we provide a solution to this problem, by showing under certain assumptions that the product of two Toeplitz operators is bounded if and only if the product of their symbols is bounded.

\end{abstract}
\section{Introduction}
The boundedness of Toeplitz operators was characterized by Otto Toeplitz in 1912 \cite{Toeplitzsoriginal1, Toeplitzsoriginal2} with his elegant characterization, showing that a Toeplitz operator is bounded if and only if its symbol is bounded (see Section 1.1 in \cite{o2021multidimensional} for a detailed history of the topic). This led to a rich active area of research at the intersection between Operator Theory and Complex Analysis. Despite more than a century of research into Toeplitz operators, there is no function-theoretic characterisation for the product of Toeplitz operators being bounded. The purpose of this paper is to provide such a characterisation, by showing under certain criteria, the product of Toeplitz operators is bounded if and only if the product of their symbols is bounded.

In 1994, Sarason posed a more specialized version of the problem in the problem book \cite{sarasonprobbook}.

\begin{Question}[Problem 7.9 \cite{sarasonprobbook}]\label{Q}
Characterize the pairs of outer functions $u, v$ in $H^{2}$ of the unit disk such that the operator $T_{u} T_{\overline{v}}$ is bounded on $H^{2}$.
\end{Question}

Over the past three decades there have been many contributions to this open question in the context of the Hardy, Bergman and Fock spaces, for example, \cite{prodH1, prodH2, prodH3, prodB1,prodB2, prodB3, prodF2, prodF3, stroethoff2002invertible} as well as many others. Although Toeplitz operators have a well-established theory, questions on Toeplitz products remain particularly challenging (see, for example, the zero Toeplitz product problem \cite{AlemanDuke}).


Characterising the boundedness of products of Toeplitz operators is equivalent to a problem in Harmonic Analysis, which asks when the two-weighted Riesz projection is bounded (see Section \ref{Harmonicanalysis} for a detailed explanation of this). Since the setup and techniques we pursue lend themselves better to function theory, we approach the problem from the perspective of Toeplitz operators and express our results in terms of the two-weighted Riesz projection in Section \ref{Harmonicanalysis}.

The main new result is Theorem \ref{mainT}, which states that when the symbols of two Toeplitz operators are an admissible Toeplitz pair (defined as Definition \ref{ATP}) the product of the Toeplitz operators is bounded if and only if the product of their symbols is bounded. In Section \ref{2} we provide background on Hardy spaces and Toeplitz operators. Section \ref{3} develops new theory on the poles of functions in the Hardy space, which is crucial for the proof of Theorem \ref{mainT}. Section \ref{4} presents the proof of Theorem \ref{mainT}. Section \ref{Harmonicanalysis} expresses new results in terms of the two-weighted Riesz projection and proves a conjecture of Sarason on Question \ref{Q} when the symbols of the operators are an admissible Toeplitz pair. Section \ref{Harmonicanalysis} also shows the product of Toeplitz operators is bounded if and only if a certain weighted weighted Lebesgue measure is a Carleson measure for the range space of a Toeplitz operator.

\section{Preliminaries}\label{2}
\subsection{Hardy Space Theory}
In order to prove Theorem \ref{mainT} one needs an understanding of the Hardy spaces, denoted $H^p(\mathbb{T})$, for $0<p \leqslant \infty$. Due the $H^p(\mathbb{T})$ spaces for $p < 1$ being less well known, we include a brief discussion of these spaces. An excellent source for such spaces (in which all the following results can be found) is \cite{cima2000backward}. For $p \in(0, \infty)$, the Hardy space $H^p(\mathbb{T})$ is the set of analytic functions $f$ on the open unit disk, $\mathbb{D}:=\{z \in \mathbb{C} :|z|<1\}$, for which the quantity
$$
\|f\|_{p}:=\sup _{0<r<1}\left\{\int_{0}^{2 \pi}\left|f\left(r e^{i \theta}\right)\right|^{p} \frac{d \theta}{2 \pi}\right\}^{1 / p}
$$
is finite. For $p \in[1, \infty)$, the quantity $\|f\|_{p}$ defines a norm on $H^p(\mathbb{T})$ which makes it a Banach space; while for $p \in(0,1)$, the metric $\rho(f, g):=\|f-g\|_{p}^{p}$ 
makes $H^p(\mathbb{T})$ a a complete topological vector space with
topology given by a translation invariant metric (Theorem 3.2.7 in \cite{cima2000backward}). For $f \in H^p(\mathbb{T})$ where $0 < p < \infty$, Theorem 3.2.3 in \cite{cima2000backward} says that the radial limit $f(\zeta):=\lim _{r \rightarrow 1^{-}} f(r \zeta)
$
exists for almost every $\zeta \in \mathbb{T}:=\{z \in \mathbb{C} :|z|=1\}$ and furthermore for $\mathrm{d}m:=\mathrm{d} \theta / 2 \pi$, the normalised Lebesgue measure on the unit circle $\mathbb{T}$,
\begin{equation}\label{obs}
    \|f\|_{p}:=\sup _{0<r<1}\left\{\int_{\mathbb{T}}|f(r \zeta)|^{p} d m(\zeta)\right\}^{1 / p}=\left\{\int_{\mathbb{T}}|f(\zeta)|^{p} d m(\zeta)\right\}^{1 / p} .
\end{equation}
With \eqref{obs}, we will freely identify a function $f$ in the $H^p(\mathbb{T})$ space with its radial boundary values on $\mathbb{T}$, thus viewing the $H^p(\mathbb{T})$ space as lying inside the space $L^p (\mathbb{T})$. $H^{2}(\mathbb{T})$ is a RKHS with reproducing kernel at $x  \in \mathbb{D}$ given by $
\Tilde{k_{x}}(z):=\frac{1}{1-\overline{x} z}$
meaning that for $f \in H^2(\mathbb{T})$, we have
$
f(x)=\left\langle f, \Tilde{k_{x}}\right\rangle_{H^{2}}
$ We say a function $h \in H^p(\mathbb{T})$ is inner if $| h | = 1$ a.e. on $\mathbb{T}$. An outer function is a holomorphic function on $\mathbb{D}$, $G $ such that $$
G(z)=c \exp \left(\int_{\mathbb{T}} \frac{\zeta +z}{\zeta-z} \log \left(\varphi\left(\zeta \right)\right) \mathrm{d}m(\zeta) \right)
$$
for some $c \in \mathbb{C}$ with $|c|=1$, and some function $\varphi$ on $\mathbb{T}$ such that $\log (\varphi)$ is integrable on $\mathbb{T}$. A classical result (which can be found as Theorem 3.2.4 in \cite{cima2000backward}) states that for $0 < p < \infty$ every $f \in H^p(\mathbb{T})$  can be factorised so that $f = f^i f^0$, where $f^i$ is an inner function and $f^o$ is an outer function. 

For a function $f \in L^p(\mathbb{T})$, where $1 \leqslant p \leqslant \infty$, define the Riesz projection $P$ by 
$$
(P f)(z):=\sum_{n=0}^{\infty} \hat{f}(n) z^{n}=\int_{\mathbb{T}} \frac{f(\zeta)}{1-\overline{\zeta} z} d m(\zeta), \quad z \in \mathbb{D},
$$
where $\hat{f}(n) :=\int_{\mathbb{T}} f(\zeta) \overline{\zeta}^{n} d m(\zeta)$. Observe that as $| \hat{f}(n) | = O(1)$ for all $ f \in L^1(\mathbb{T})$, $P(f)$ defines an analytic function in $\mathbb{D}$. 

\begin{Definition}
    $H(\mathbb{T})$ is the space of holomorphic functions on $\mathbb{D}$ that have $a.e.$ defined radial boundary values on $\mathbb{T}$. We equip this space with the topology of uniform convergence on compact subsets, i.e., $f_n \rightarrow f_0$ in $H(\mathbb{T})$ if for all compact $K \subseteq \mathbb{D}$, $f_n \rightarrow f_0$ uniformly on $K$.
\end{Definition}
We remark that for all $0<p$, $H^p (\mathbb{T})  \subseteq N^+\subseteq N \subseteq H(\mathbb{T})$, where $N^+$ is the Smirnov class and $N$ is the Nevanlinna class.

For a set of analytic functions on the disc $X$, we will the write $\overline{X}$ to mean the conjugate of all functions in $X$, i.e. $\overline{X} =  \{ \overline{f(z)} :  f(z) \in X \}$. For $1<p< \infty$, Section 3.3 in \cite{cima2000backward} shows that the Riesz projection is a bounded operator $P: L^p(\mathbb{T}) \to H^p(\mathbb{T})$, and consequently for $1< p < \infty$, we can write $L^p(\mathbb{T}) = \overline{ z H^p(\mathbb{T})} \oplus H^p(\mathbb{T})$. We also note that Section 3.3 in \cite{cima2000backward} demonstrates $P (L^1(\mathbb{T})) \neq H^1(\mathbb{T})$. Throughout, we use the notation $B_{\epsilon}(y)$ to mean to open ball radius $\epsilon$ centred at $y$. 

\subsection{Toeplitz operators}

\begin{Definition}
    The Toeplitz operator with symbol $g \in L^2(\mathbb{T})$ is 
    the mapping $f \mapsto P(gf)$ (where here the multiplication $gf$ is understood as multiplication of the functions $f$ and $g$ on $\mathbb{T}$).
\end{Definition}

Let $u \in H(\mathbb{T}), v \in L^2(\mathbb{T})$ and let $0 <s<1$. Remark 3.4.3 in \cite{cima2000backward} shows that $P: L^1(\mathbb{T}) \to H^s(\mathbb{T})$ is continuous and consequently the map  $\Tilde{T_{{v}}}: H^2  \to H^s(\mathbb{T})$ such that $\Tilde{T}_{v} f = P ( {v} f)$ is continuous. Define $\Tilde{T_u} : H^s(\mathbb{T}) \to H(\mathbb{T})$ by $f \mapsto uf$. As $H^s(\mathbb{T})$ convergence implies convergence in $H(\mathbb{T})$ (see Proposition 3.2.6 in \cite{cima2000backward}) $\Tilde{T_u}$ is also continuous, so the map \begin{equation}\label{Ttilde}
    \Tilde{T} := \Tilde{T_u} \Tilde{T_{{v}}} :  H^2 \to H(\mathbb{T})
\end{equation} is continuous.

If $g \in L^p(\mathbb{T}) \setminus L^2({\mathbb{T}})$ for some $p<2$, then the Toeplitz operator $T_g$ is no-longer well defined, since for $f \in H^2(\mathbb{T})$, $gf$ may not lie in $L^1(\mathbb{T})$ and so $P(gf)$ is not well defined. Similarly, when one considers the product $T_hT_{g}$ for $g \in L^2(\mathbb{T})$  if $h$ was not analytic, the product $T_h T_g(f) = P(hP(gf))$ would not be well-defined as $hP(gf)$ will not always lie in $L^1(\mathbb{T})$. In this sense the conditions $u \in H(\mathbb{T})$ and $v \in L^2(\mathbb{T})$ are the most general symbols for which one can characterise boundedness of $T_u T_v$ in Theorem \ref{mainT}.

Although the following application of the Closed Graph Theorem is standard when considering single Toeplitz operators (i.e. not a product of Toeplitz operators), we show the product of two Toeplitz operators is bounded whenever the range of the product is contained in $H^2(\mathbb{T})$.

\begin{Lemma}\label{C}
Let $u \in H(\mathbb{T})$ and $v \in L^2(\mathbb{T})$. Then 
$T_u T_{{v}}$ is a bounded map from $H^2(\mathbb{T})$ to $H^2(\mathbb{T})$ if and only if $\ran \Tilde{T} \subseteq H^{2}(\mathbb{T})$, where $\Tilde{T}$ is given by \eqref{Ttilde}.
\end{Lemma}
\begin{proof}
When $T_u T_{{v}}$ is a bounded map from $H^2(\mathbb{T})$ to $H^2(\mathbb{T})$, as $\Tilde{T} (f) = T_u T_{\overline{v}}(f)$, clearly $\ran \Tilde{T} \subseteq H^{2}(\mathbb{T})$

Let $\ran \Tilde{T} \subseteq H^2(\mathbb{T})$, then $T_u T_v : H^{2}(\mathbb{T})\to H^2(\mathbb{T})$ is a well-defined map. Let $(f_n)_{n \in \mathbb{N}} \in H^2$ be such that $f_n \overset{H^2}{\to} y_1$ and $T_u T_v (f_n) \overset{H^2}{\to} y_2$. Continuity of $T_u T_v:  H^2(\mathbb{T}) \to H(\mathbb{T})$ implies $T_uT_v (f_n) \overset{H(\mathbb{T})}{\to} T_u T_v (y_1)$, and as $H^2(\mathbb{T})$ convergence is stronger than $H(\mathbb{T})$ convergence, $T_u T_v (f_n) \overset{H(\mathbb{T})}{\to} y_2$, thus $T_u T_v (y_1) = y_2$ and by the Closed Graph Theorem, $T_u T_v : H^2(\mathbb{T}) \to H^2(\mathbb{T})$ is bounded. 
\end{proof}

\section{Poles of $H(\mathbb{}T)$ Functions}\label{3}

\begin{Definition}
    We say that a point $\zeta \in \mathbb{T}$ is a pole of $f \in H(\mathbb{T})$ if for all $\epsilon>0$, $f \notin L^{\infty}(B_{\epsilon}(\zeta) \cap \mathbb{T}).$ For $\overline{v_-} + v_+ = v \in L^2(\mathbb{T})$, where $v_+ \in H^2, \overline{v_-} \in \overline{z H^2(\mathbb{T})}$ the poles of $v$ are defined as the union of the poles of $v_-$ and $v_+$.
\end{Definition}

\begin{Definition}
    For $ f \in H(\mathbb{T})$ with a pole at $\zeta \in \mathbb{T}$, we say $f$ has a simple pole at $\zeta$ if there is a $n \in \mathbb{N}$ such that $(z-\zeta)^n f(z) \in L^{\infty}(\mathbb{T} \cap B_{\epsilon}(\zeta))$ for some $\epsilon >0 $. The smallest such $n \in \mathbb{N}$ such that $(z-\zeta)^n f(z) \in L^{\infty}(\mathbb{T} \cap B_{\epsilon}(\zeta))$ for some $\epsilon > 0$ is the order of the pole.
\end{Definition}


If $f$ is meromorphic on $\mathbb{D}$ with finitely many poles, then for a pole $z_0 \in \mathbb{D}$, $ |f(z)| \rightarrow  \infty$ as $ {z \rightarrow z_0} $. In contrast to this, when we characterise the poles of functions in $H^2(\mathbb{T})$ there may be some oscillatory behaviour of functions as you approach the pole along $\mathbb{T}$. Part $(b)$ in the following proposition demonstrates this oscillatory behaviour on $\mathbb{T}$.

\begin{Proposition}
    \begin{enumerate}
        \item There exist functions in the Hardy space with poles of all orders.
        \item There exist functions $f \in H^2(\mathbb{T})$ such that $f$ is analytic on $\mathbb{T} \setminus \{1\}$ and sequences $b_k,  c_k \in \mathbb{T}^+ : = \{ z \in \mathbb{T} :  0 <\arg z < \pi/2\}$ converging to 1 such that $|b_{k+1} - 1| < |c_k - 1| < |b_{k} - 1|$ for all $k \in \mathbb{N}$, where $f(c_k) =0$,  $f(b_k) \to \infty$.
    \end{enumerate}
\end{Proposition}

\begin{proof}
Let $\theta$ be a Blaschke product with zero set $(a_k)$ converging to the point 1. As $\theta^'$ is unbounded on $\mathbb{T}$, $\theta'$ must be unbounded on $\mathbb{T}^+$ or $\mathbb{T}^- = \{ z \in \mathbb{T} :  -\pi/2 <\arg z < 0\}\}$. We assume WLOG that $\theta'$ is unbounded on $\mathbb{T}^+ $ (since otherwise the inner function $z \mapsto \overline{\theta(\overline{z})}$ will be unbounded on $\mathbb{T}^+$). Then $\arg \theta$ (viewed as a map into $(0, \infty)$) is monotonic and unbounded on $\mathbb{T}^+$. Thus for any $t \in \mathbb{T}$, there exists a sequence $x_k \in \mathbb{T}^+$, converging to 1 such that $\theta(x_k) = t$.

    $(a)$ Consider the function $f(z) = \frac{1 - \overline{\theta(1)} \theta(z)}{1 - z}$. Then \cite{cohn1986radial} shows $f \in H^p(\mathbb{T})$ if and only if $$\sum_k \frac{1-\left|a_k\right|^2}{\left|1-a_k\right|^p} < \infty.$$ By picking $a_k$ such that the above holds for $p = 2^{-n}$ this guarantees $\theta(1)$ is unimodular (see for example \cite{ahern1970radial}) and since $f \in H^{2^{-n}}(\mathbb{T})$, we have $f^n = \left( \frac{1 - \overline{\theta(1)} \theta(z)}{1 - z}\right)^n \in H^2(\mathbb{T})$ and $(z-1)^n f^n(z) = (1 - \overline{\theta(1)} \theta(z)) \in H^{\infty}(\mathbb{T})$. For any $i < n$ the function $(z-1)^i f^n(z) = \frac{(1 - \overline{\theta(1)} \theta(z))^n}{(1 - z)^{n-i}} \notin L^{\infty}(\mathbb{T} \cap B_{\epsilon}(1))$ for any $\epsilon>0$ since there is a sequence of points $b_k$ on $\mathbb{T}^+$ converging to 1 such that $\theta(b_k) = - \overline{\theta(1)}$. Thus $f^n$ has a pole of order $n$ at the point 1.

    $(b)$ Let $c_k \in \mathbb{T}^+$ be a sequence converging to 1 such that $\theta(c_k) = \overline{\theta(1)}$. By passing to a subsequence of $b_k$ and $c_k$ if necessary we see $f(z) = \frac{1 - \overline{\theta(1)} \theta(z)}{1 - z}$ has the specified properties.
\end{proof}

The following example shows we do not always have oscillatory behaviour of $f \in H^2(\mathbb{T})$ as we approach its pole.

\begin{Example}
    $f(z) = \frac{1}{(1-z)^{1/3}} \in H^2(\mathbb{T})$ and has a pole of order 1 at 1.
\end{Example}





In fact for $f \in H^2(\mathbb{T})$, the following proposition shows boundedness of $(z- \zeta)^n f(z)$ on some open subset of $\mathbb{T}$ containing $\zeta$ is equivalent to $(z- \zeta)^n f(z)$ not oscillating as you approach $\zeta$ along $\mathbb{T}$. In the following, we use the notation $B_{\epsilon}^+(\zeta) := B_{\epsilon}(\zeta)  \cap \{ z \in \mathbb{T} :  \arg \zeta < \arg z < \arg \zeta + \pi/2  \}$.


\begin{Proposition}
    Let $f \in H^2(\mathbb{T})$ have a pole at $\zeta \in \mathbb{T}$ and be well-defined on $B_{\epsilon}^+(\zeta) \setminus \{\zeta \}$ and let $n \in \mathbb{N}$. Then $(z- \zeta)^n f(z) \notin L^{\infty}(B_{\epsilon}^+(\zeta))$ for all $\epsilon >0$ if and only if 
    there exists sequences $x_k, y_k \in B_{\epsilon}^+(\zeta)$ such that  $x_k \rightarrow \zeta$, $y_k \rightarrow \zeta$ with $|x_{k+1} - \zeta| < |y_k - \zeta| < |x_k - \zeta|$ for every $k$ and $(x_k- \zeta)^n f(x_k) \rightarrow 0$, $(y_k- \zeta)^n f(y_k) \rightarrow \infty$. 
\end{Proposition}

\begin{proof}
    The backward implication is immediate, so we prove the forward implication. 
    Suppose for contradiction that there are no sequences $x_k \in B_{\epsilon}^+(\zeta)$ such that $x_k \rightarrow \zeta$ and $(x_k- \zeta)^n f(x_k) \rightarrow 0$. Then there exists $0 < c, \, 0< \delta < \epsilon $ such that $c \leqslant|f(z) ( \zeta - z)^n|$ on $B_{\delta}^+(\zeta)$, so
    $$
    \infty = c \int_{B_{\delta}^+(\zeta) }\frac{1}{|1-\overline{\zeta}z|^{2n}} \mathrm{d}m= c \int_{B_{\delta}^+(\zeta)}\frac{1}{|\zeta-z|^{2n}} \mathrm{d}m \leqslant  \int_{B_{\delta}^+(\zeta)} |f(z)|^2 \mathrm{d}m < \infty.
    $$
    
Thus, such a sequence $x_k$ must exist and clearly as $(z- \zeta)^n f(z) \notin L^{\infty}(B_{\epsilon}^+(\zeta))$ a sequence $y_k$ with the specified properties exists. The result now follows by passing to a subsequence of both $x_k$ and $y_k$ if necessary. 
\end{proof}

\begin{Remark}
 In the proposition above, we may omit the condition that $f$ is well defined on $B_{\epsilon}^+(\zeta) \setminus \{ \zeta \}$ if we work with disjoint open subarcs $X_k, Y_k \subseteq B_{\epsilon}^+(\zeta) \setminus \{ \zeta \}$ such that $\operatorname{dist}( X_k , \zeta) \rightarrow 0$, $\operatorname{dist}( Y_k , \zeta) \rightarrow 0$, $\|(z- \zeta)^n f (z)\|_{L^{\infty}(Y_k)} \rightarrow \infty, \, \|(z- \zeta)^n f (z)\|_{L^{\infty}(X_k)} \rightarrow 0$ instead of sequences $x_k , y_k \in \mathbb{T}$.
\end{Remark}

It may be tempting to believe that all poles of functions in $H^2(\mathbb{T})$ are simple, or that every function in $H^2(\mathbb{T})$ only has finitely many poles. However, the following proposition shows this is not the case. 

\begin{Proposition}
\begin{enumerate}
    \item There exists a function $f \in H^2 ( \mathbb{T})$ with a pole at $1$ such that for all $n \in \mathbb{N}$, $f(z) (z-1)^n \notin L^{\infty}(B_{\epsilon}(1) \cap \mathbb{T})$ for all $\epsilon >0$.
    \item There exists functions in $H^2(\mathbb{T})$ with infinitely many poles on $\mathbb{T}$.
\end{enumerate}
    
\end{Proposition}

\begin{proof}
    \begin{enumerate}
        \item Let $\epsilon > 0$ and $n \in \mathbb{N}$. For $f \in L^2(\mathbb{T})$ 
        since $ f = \overline{f_-}  + f_+$ where $f_+ \in H^2(\mathbb{T}), \, \overline{f_-}  =  \overline{z H^2(\mathbb{T})}$, if $f(z) (z-1)^n \notin L^{\infty}(B_{\epsilon}(1) \cap \mathbb{T})$ this forces either $f_+(z) (z-1)^n \notin L^{\infty}(B_{\epsilon}(1) \cap \mathbb{T})$  or $f_-(z) (z-1)^n \notin L^{\infty}(B_{\epsilon}(1) \cap \mathbb{T})$. Thus if  $f(z) (z-1)^n \notin L^{\infty}(B_{\epsilon}(1) \cap \mathbb{T})$ for every $n \in \mathbb{N}$, there is an infinite sequence $(n_j) \subseteq  \mathbb{N}$, such that either $f_+(z) (z-1)^{n_j} \notin L^{\infty}(B_{\epsilon}(1) \cap \mathbb{T})$ for all $j$, or $f_-(z) (z-1)^{n_j} \notin L^{\infty}(B_{\epsilon}(1) \cap \mathbb{T})$ for all $j$. If for all $j$, $f_+(z) (z-1)^{n_j} \notin L^{\infty}(B_{\epsilon}(1) \cap \mathbb{T})$, this implies for all $n \in \mathbb{N}$, $f_+(z) (z-1)^{n} \notin L^{\infty}(B_{\epsilon}(1) \cap \mathbb{T})$. Similarly, if $f_-(z) (z-1)^{n_j} \notin L^{\infty}(B_{\epsilon}(1) \cap \mathbb{T})$, then $f_-(z) (z-1)^{n} \notin L^{\infty}(B_{\epsilon}(1) \cap \mathbb{T})$.

    To construct such an $f \in L^2(\mathbb{T})$ for $k \in \mathbb{N}$ $k>2$ we set $ x_k \in \mathbb{T}$ such that $|x_k -1| = \frac{1}{k}$, we set $X_k \subseteq \mathbb{T}$ to be an arc of $\mathbb{T}$ centred at $x_k$ with $m (X_k) = \frac{1}{k!k}$. Set $f$ to take value $k!$ on $X_k$ and be zero elsewhere. Then since $X_i \cap X_j$ is empty for $i \neq j$ we have $\|f\|_{L^2(\mathbb{T})} = \sum_{k>2} \left( \frac{1}{k} \right)^2 < \infty$ but for every $n \in \mathbb{N}$, we have $|f(x_k) (x_k - 1)^n| = \frac{k!}{k^n} \rightarrow \infty$ as $k \rightarrow \infty$. 
    \item Writing $f = \overline{f_-} + f_+$ shows if there exists an $f \in L^2(\mathbb{T})$ such that for infinitely many $t_i \in \mathbb{T}$, we have $f \notin L^{\infty}(B_{\epsilon}(t_i) \cap \mathbb{T})$ for all $\epsilon >0$, then there are functions in $H^2(\mathbb{T})$ with infinitely many poles on $\mathbb{T}$. To construct $f \in L^2(\mathbb{T})$ with this property, for $k \in \mathbb{N}$ consider a sequence of disjoint subsets $X_k \subseteq \mathbb{T}$ and functions $f_k : X_k \to \mathbb{C} $ defined a.e. on $X_k$ with a pole at a point $x_k \in X_k$ such that $\|f_k\|_{L^2(X_k)} = \frac{1}{k^2}$. Then $f(z) = \sum_{k \in \mathbb{N}}f_k (z) \in L^2(\mathbb{T})$ has the required property.
    \end{enumerate}  
\end{proof}

\section{Boundedness of Products of Toeplitz Operators}\label{4}

\begin{Definition}\label{ATP}
    We say $u \in H(\mathbb{T}), v \in L^2(\mathbb{T})$ are an admissible Toeplitz pair if the following conditions hold.
    \begin{enumerate}
        \item Both $u,v$ have finitely many poles, with each pole being simple.
        \item $u$ is analytic at the poles of $v$.
        \item When $v$ is written as $v = \overline{v_-} + v_+$ with $\overline{v_-} \in \overline{z H^2(\mathbb{T})}, \, v_+ \in H^2(\mathbb{T})$, both $v_-$ and $v_+$ are analytic at the poles of $u$.
        \item One of the following holds: \begin{itemize}
            \item[(1)] $v \in \overline{H^2(\mathbb{T})} $,
           \item[(2)] or for all $t \in \mathbb{T}$ which is a pole of $u$ of order $j$, for all $ 0 \leqslant i <j$ whenever $\frac{v(z)}{(z-t)^{i}} (t) = 0$ we have $\frac{v_+(z)}{(z-t)^i}(t) =  \frac{\overline{v_-(z)}}{(z-t)^i}(t) = 0$.
        \end{itemize}  
    \end{enumerate}
\end{Definition}
We remark that in the context of Question \ref{Q} (i.e. when $u,\overline{v} \in H^2(\mathbb{T})$) condition $(d)$ of Definition \ref{ATP} automatically holds.

\begin{Theorem}\label{mainT}
    Let $u \in H(\mathbb{T}), v \in L^2(\mathbb{T})$ be an admissible Toeplitz pair. Then $T_u T_{{v}}$ is bounded on $H^2(\mathbb{T})$ if and only if $u {v} \in L^{\infty}(\mathbb{T})$.
\end{Theorem}
We prove the result in the case $(d) (2)$ holds. When $(d) (1)$ holds, one can make clear simplifications of the following proof.
\begin{proof}
We extend the argument from \cite[Problem 7.9]{sarasonprobbook} to show the forward implication. \cite[Corollary 5.22]{paulsen2016introduction} shows $T_{\overline{v_-}}(k_x)= \overline{v_-(x)}k_x$. When $\|T_u T_{{v}}\| \leqslant M$, for all normalised reproducing kernels $k_x \in H^2(\mathbb{T})$,
$$
|u(x){v}(x)| = |\langle  (T_{\overline{v_-}}(k_x) + v_+k_x, k_x>| =   |u(x)||\langle  T_{v} k_x, k_x>|  = |\langle T_u T_{v} k_x, k_x>| \leqslant M.
$$

Let $uv \in L^{\infty}(\mathbb{T})$. Let $E := \{ t_1, t_2, \ldots t_{N} \}$ (respectively $F := \{ s_1, s_2, \ldots s_{M} \}$) be the poles of $u$ (respectively $v$). By Lemma \ref{C} $T_{u}T_{{v}}$ is bounded if and only if for some $ \epsilon >0$ and all $f \in H^2(\mathbb{T})$
    \begin{equation}\label{main}
        \int_{\mathbb{T} \setminus B_{\epsilon}(E \cup F)}|uT_{v}(f)|^2 \mathrm{d}m + \int_{ B_{\epsilon}(F) \cap \mathbb{T}}|uT_{v}(f)|^2 \mathrm{d}m + \int_{ B_{\epsilon}(E ) \cap \mathbb{T}}|uT_{v}(f)|^2 \mathrm{d}m < \infty.
    \end{equation}

As $u {v} \in L^{\infty}$ the poles of $u$ and poles of ${v}$  must be distinct. Let $k_i$ (respectively $n_i$) be the order of the pole $t_i$ (respectively $s_i)$. As $u {v} \in L^{\infty}(\mathbb{T})$ this forces $u(s_i) = 0$, and as $u$ is analytic at $s_i$, $ \frac{u}{z-s_i}$ is analytic at $s_i$. If $n_i >1$, then $\frac{u}{z-s_i} {v} (z-s_i) \in L^{\infty}(\mathbb{T})$, which means $\frac{u}{z-s_i}$ has a zero at $s_i$. Recursively we see $\frac{u}{(z-s_i)^{n_i}}$ is  analytic at $s_i$. Similarly for each $t_i$, $uv \in L^{\infty}(\mathbb{T})$ means $v(t_i) = 0$, thus condition $(d)(2)$ of Definition \ref{ATP} implies $\frac{v_+}{z-t_i}$ and $\frac{v_-}{z-t_i}$ are analytic at $t_i$. If $k_i > 1$ we recursively deduce $\frac{v_+}{(z-t_i)^{k_i}}$ and $\frac{v_-}{(z-t_i)^{k_i}}$ are analytic at $t_i$.

Pick $\epsilon$ sufficiently small such that \begin{enumerate}
    \item $B_{\epsilon}(E \cup F)$ consists of finitely many non intersecting balls.
    \item At each $t_i \in E$, both $\frac{v_+}{(z-t_i)^{k_i}} \in L^{\infty}(B_{\epsilon}(t_i) \cap \mathbb{T})$ and $\frac{v_-}{(z-t_i)^{k_i}} \in L^{\infty}(B_{\epsilon}(t_i) \cap \mathbb{T})$ and at each $s_i \in F$, $ \frac{u}{(z-s_i)^{n_i}} \in L^{\infty}(B_{\epsilon}(s_i) \cap \mathbb{T})$.
    \item At each $t_i \in E$, we have $u(z) ( z-t_i)^{k_i} \in L^{\infty}(B_{\epsilon}(t_i) \cap \mathbb{T})$ and for each $s_i \in F$ we have $v(z) ( z-s_i)^{n_i} \in L^{\infty}(B_{\epsilon}(s_i) \cap \mathbb{T})$.
\end{enumerate} 

We now show the first integral in \eqref{main} is finite. We have $v \prod_{i \in F} (z-s_i)^{n_i} \in L^{\infty}(\mathbb{T})$, so for $f \in H^2(\mathbb{T})$, $f {v} \prod_{i \in F} (z-s_i)^{n_i} \in L^2(\mathbb{T})$, and so \begin{equation}\label{R1}
    P(f {v} \prod_{i \in F} (z-s_i)^{n_i})  = h
\end{equation} for some $h \in H^2(\mathbb{T})$. For $n \in \mathbb{N}$, observe $P(f {v} z^n) = {z}^n P({v}f) + {q_n}$ where $q_n$ is a polynomial of order $n$. Consequently \eqref{R1} becomes
\begin{equation}\label{Rbetter}
    q + \prod_{i \in F}(  (z-s_i)^{n_i}) P( {v}f) = P(f  {v} \prod_{i \in F}(  (z-s_i)^{n_i}) ) = h,
\end{equation}
where $q$ is a polynomial. Thus
\begin{equation}\label{R3}
    P( {v}f) = \frac{h-q}{\prod_{i \in F}(  (z-s_i)^{n_i}) }
\end{equation}
which lies in $L^2(\mathbb{T}\setminus B_{\epsilon}( E \cup F))$ since the numerator is in $L^2(\mathbb{T})$ and the denominator is bounded on $\mathbb{T}\setminus B_{\epsilon}(  E \cup F)$. So $$\int_{\mathbb{T} \setminus B_{\epsilon}(E \cup F)}|uT_{v}(f)|^2 \mathrm{d}m < \sup_{\mathbb{T} \setminus B_{\epsilon}(E \cup F)}|u(z)|^2  \int_{\mathbb{T} \setminus B_{\epsilon}(E \cup F)}|T_{v}(f)|^2 \mathrm{d}m < \infty.$$


As $F$ is finite, in order to show the second integral in \eqref{main} is finite it suffices to prove that for each $f \in H^2(\mathbb{T})$ and $s_i \in F$ we have $\int_{B_{\epsilon}(s_i)\cap \mathbb{T}}|uT_{v}(f)|^2 \mathrm{d}m< \infty$. From \eqref{R3},
we have
$$
uP(\overline{v}f) =\frac{u}{(z-s_i)^{n_i}} \frac{h-q}{\prod_{j \neq i }(z-s_j)^{n_j}}
$$
which must lie in $L^2(B_{\epsilon}(s_i)\cap \mathbb{T})$ since $\frac{u}{(z-s_i)^{n_i}}, \,  \frac{1}{\prod_{j \neq i }(z-s_j)^{n_j}} \in L^{\infty}(B_{\epsilon}(s_i)\cap \mathbb{T})$.

In order to prove $\int_{ B_{\epsilon}(E )}|uT_{v}(f)|^2 \mathrm{d}m < \infty$ we show $\int_{ B_{\epsilon}(s_i )}|uT_{v}(f)|^2  \mathrm{d}m< \infty$ for $s_i \in E$. 
Observe that \begin{equation}\label{R5}
    \frac{v}{\prod_{i \in F}(z-t_i)^{k_i}}\prod_{j \in F}({z}- {s_j})^{n_j}
\end{equation}
is bounded a.e. on $ B_{\epsilon}(E) \cap \mathbb{T}$ since $\frac{v}{\prod_{i \in F}(z-t_i)^{k_i}}$ and $\prod_{j \in F}({z}- {s_j})^{n_j}$ are bounded a.e. on $ B_{\epsilon}(E) \cap \mathbb{T}$ and similarly \eqref{R5} is bounded a.e. on $ B_{\epsilon}(F) \cap \mathbb{T}$ since $v \prod_{j \in F}({z}- {s_j})^{n_j}$ and $\frac{1}{\prod_{i \in F}(z-t_i)^{k_i}}$ are both bounded a.e. on $ B_{\epsilon}(F) \cap \mathbb{T}$. Since $\frac{v}{\prod_{i \in F}(z-t_i)^{k_i}}\prod_{j \in F}({z}- {s_j})^{n_j}$ may only have poles in  $E \cup F$ or, we further conclude that $\frac{v}{\prod_{i \in F}(z-t_i)^{k_i}}\prod_{j \in F}({z}- {s_j})^{n_j} \in L^{\infty}(\mathbb{T})$. So
$$
P( {v}f \prod_{j \in F}( {z} -  {s_j})^{n_j}) = P( (  {z} -  {t_i})^{k_i}g)
$$
for some $g \in L^2(\mathbb{T})$. A similar reasoning used to deduce \eqref{Rbetter} shows
$$
\prod_{j \in F}( {z} -  {s_j})^{n_j}P( {v}f) +  {q_1} = P(g)(  {z} -  {t_i})^{k_i} + {q_2}
$$ and so
\begin{align}
    &\int_{B_{\epsilon}(t_i)\cap \mathbb{T}} |u T_{v}(f)|^2  \mathrm{d}m \leqslant \\
    &\ \int_{B_{\epsilon}(t_i)\cap \mathbb{T}} \left|\frac{ P(g) u(  {z} -  {t_i})^{k_i}}{\prod_{j \in F}( {z} -  {s_j})^{n_j}}\right|^2 \mathrm{d}m+ \int_{B_{\epsilon}(t_i)\cap \mathbb{T}}\left| \frac{ {q_1}- {q_2}}{\prod_{j \in F}( {z} -  {s_j})^{n_j}}\right|^2 \mathrm{d}m.
\end{align}
To see the quantity above is finite observe that the first integral is finite since $P(g) \in L^2(\mathbb{T})$ and $\frac{ u(  {z} -  {t_i})^{k_i}}{\prod_{j \in F}( {z} -  {s_j})^{n_j}} \in L^{\infty}(B_{\epsilon}(t_i)\cap \mathbb{T})$ and the second integral is finite since it is the integral of a bounded function on $B_{\epsilon}(t_i)\cap \mathbb{T}$.
\end{proof}


\section{The Two-Weighted Riesz Projection}\label{Harmonicanalysis}
This section shows the product of Toeplitz operators being bounded is equivalent to a famous problem in Harmonic Analysis, which asks when the two-weighted Riesz projection is bounded.

In \cite{prodH1} it was first observed that one has the following commutative diagram.

\begin{equation}\label{equiv}
    \begin{array}{ccccc}
& H^2(\mathbb{T}) & \xrightarrow{T_u T_v} & H^2(\mathbb{T}) & \\
M_{v} & \downarrow & & \uparrow & M_u \\
& L^2\left(\frac{1}{|v|^2}, \mathbb{T}\right) & \xrightarrow{P_R} & H^2(\mathbb{T})\left(|u|^2, \mathbb{T}\right) &
\end{array}
\end{equation}
Here, $ P_R: L^2\left(\mathbb{T}, \frac{1}{|v|^2}\right) \rightarrow L^2\left(\mathbb{T},|u|^2\right)$ is the two-weighted Riesz projection and $M_u, M_v$ on the vertical sides denote multiplication with the respective symbols which are isometric by definition of the weights. Hence, boundedness of $ P_R: L^2\left(\frac{1}{|v|^2}, \mathbb{T} \right) \rightarrow L^2\left(|u|^2, \mathbb{T} \right)$ is equivalent to boundedness of the product $T_u T_v$ acting on $H^2(\mathbb{T})$. 

The problem of classifying those pairs of weights $(\rho_1, \rho_2)$ for which the two-weighted Riesz projection
\begin{equation}\label{hilbertT}
    P_R: L^2(\mathbb{T}, \rho_1) \rightarrow L^2(\mathbb{T}, \rho_2 )
\end{equation}
or equivalently, the two-weighted Hilbert transform, is bounded has a rich history within Harmonic Analysis. It was conjectured that the condition
\begin{equation}\label{condition}
    \sup _{|w| <1} {P_w}(\rho_1){P_w}\left(\rho_2^{-1}\right)(z)<\infty,
\end{equation}
where $P_w(\rho_1)$ is the Poisson extension of $\rho_1$ at $w$,
characterises \eqref{hilbertT}. 
However, Nazarov disproved this conjecture 
in 1997 \cite{prodH3}. The problem of characterising boundedness
of \eqref{hilbertT} has been the subject of intense recent research activity, see e.g. \cite{HarmA1, HarmA2, HarmA3, HarmA4, HarmA5} and the references therein.

Using the equivalence highlighted in \eqref{equiv} and Theorem \ref{mainT} one deduces the following.

\begin{Theorem}
Let $u \in H(\mathbb{T}),v \in L^2(\mathbb{T})$ be an admissible Toeplitz pair. Then the two-weighted Riesz projection
(1.13) $P_R: L^2(\mathbb{T}, \frac{1}{|v|^2}) \rightarrow L^2(\mathbb{T}, |u|^2)$ is bounded if and only if $uv \in L^{\infty}(\mathbb{T})$.
\end{Theorem}

Motivated by the conjectured condition \eqref{condition} there was also was a conjectured answer to Question \ref{Q}. In \cite{sarasonprobbook} Treil showed that for $u, v \in H^2(\mathbb{T})$ outer functions, if the product $T_u T_{\overline{v}}$ is bounded, then \begin{equation}\label{Poisson}
    \sup _{|w|<1} P_w\left(|u|^2\right) P_w\left(|v|^2\right)<\infty.
\end{equation} 
Sarason conjectured in \cite{sarasonprobbook} that \eqref{Poisson} is sufficient for boundedness of $T_u T_{\overline{v}}$ on $H^2(\mathbb{T})$. Nazarov's counterexample \cite{prodH3} to the conjecture that \eqref{condition} is equivalent to the mapping in \eqref{hilbertT} being bounded consequently also showed that \eqref{Poisson} is not sufficient for boundedness of $T_uT_{\overline{v}}$ on $H^2(\mathbb{T})$. However, the following corollary shows that Sarason's conjecture is true when $u,v$ are an admissible Toeplitz pair.

\begin{Corollary}\label{5.2}
Let $u, v \in H^2(\mathbb{T})$ be outer functions which are an admissible Toeplitz pair. Then $T_uT_{\overline{v}}$ is bounded on $H^2(\mathbb{T})$ if and only if \eqref{Poisson} holds.
\end{Corollary}

\begin{proof}
The forward implication is shown in \cite[Problem 7.9]{sarasonprobbook} so we only prove the backward implication. It is well known that for a.e.$ t \in \mathbb{T}$, as ${w \rightarrow t}$ radially, $ P_w\left(|u|^2\right) \rightarrow |u|^2 (t) $. Thus if $\sup _{|w|<1} P_w\left(|u|^2\right) P_w\left(|v|^2\right) \leqslant M$, then for a.e. $ t \in \mathbb{T}$
$$
|u(t) v(t)| = \sqrt{|u|^2 (t) |v|^2 (t)} = \sqrt{\lim_{w \to t \, \, \text {(radially)}} P_w\left(|u|^2\right)P_w\left(|v|^2\right)} \leqslant \sqrt{M}.
$$
The result now follows from Theorem \ref{mainT}.
\end{proof}

\subsection{Range Spaces}
If $\mu $ is a positive finite Borel measure on the closed unit disk and $\mathcal{H}$ is a Hilbert space of analytic functions on $\mathbb{D}$ we say that $\mu$ is a Carleson measure for $\mathcal{H}$ when there exists a $C>0$ such that $
\|f\|_\mu \leqslant C \|f\|_{\mathcal{H}}
$
for all $f \in \mathcal{H}$.

In this subsection we show boundedness of $T_uT_{{v}}$  is also equivalent to $|u|^2 dm$ being a Carleson measure for the range space of $T_{{v}}.$ Range spaces were formally introduced by Sarason \cite{sarason1994sub}, and range spaces of co-analytic Toeplitz operators with bounded symbols were studied in \cite{rangespaces, fricain2019multipliers, mccarthy1990common}. We first generalise the notion of range spaces of co-analytic Toeplitz operators to unbounded symbols. 

\begin{Lemma}
Let $v \in H^2(\mathbb{T})$ be outer and $0<s<1$. The bounded map $\Tilde{T_{\overline{v}}}: H^2(\mathbb{T}) \to H^s(\mathbb{T})$, where $\Tilde{T}_{\overline{v}}(f) = P(\overline{v}f)$ is injective. 
\end{Lemma}
\begin{proof}
We first show the range of $P_- := I_d - P : L^1(\mathbb{T}) \to L^s(\mathbb{T})$ is contained in $\overline{ z H^s(\mathbb{T})}$. Given $f_0 \in L^1(\mathbb{T})$, let $( f_n)_{n \in \mathbb{N}}$ be a sequence in $L^2(\mathbb{T})$ such that $f_n \overset{L^1(\mathbb{T})}{\to} f_0$. As $P_-(f_n) \in \overline{z H^2} $, by continuity of $P_- : L^1(\mathbb{T}) \to L^s(\mathbb{T})$ we must have that $P_-(f_0) $ lies in the closure of $\overline{z H^2}$ in the $L^s(\mathbb{T})$ metric. However, Theorem 3.2.7 in \cite{cima2000backward} shows that $H^s(\mathbb{T})$ is closed in $L^s(\mathbb{T})$, and as $H^2(\mathbb{T})$ is dense in $H^s(\mathbb{T})$ (claim 1 in Proposition 6.1.6 in \cite{cima2000backward}) it follows that the closure of $\overline{z H^2}$ in the $L^s(\mathbb{T})$ norm is $\overline{z H^s(\mathbb{T})}$. Thus $P_-(f_0) \in \overline{z H^s(\mathbb{T})}$, and as our choice of $f_0 \in L^1(\mathbb{T})$ was arbitrary it follows that $\ran P_- \subseteq \overline{z H^s(\mathbb{T})}$.

The Smirnov class, denoted $N^+$, may be defined by $$
N^+ =  \{ \frac{f_1}{f_2} : f_2 \text{ is outer  } f_1, \, f_2 \in H^s(\mathbb{T}) \}.
$$
Proposition 2.2 in \cite{toeplitzkernelbackwardarticle} shows that $N^+ \cap L^p(\mathbb{T}) = H^p(\mathbb{T})$ for $0 < p < \infty$. Thus as $v$ is outer, if $f \in \ker \Tilde{T_{\overline{v}}}$, we have $f \overline{v} =P_- ( f \overline{v})  \in \overline{z H^s(\mathbb{T})}$, and so $f \in \overline{z N^+} \cap H^2 = \overline{z N^+} \cap L^2 \cap H^2 = \overline{ z H^2} \cap H^2 = \{ 0 \}$.
\end{proof}

\begin{Proposition}\label{D}
Let $v \in H^2(\mathbb{T})$ be outer. The space $\ran \Tilde{T_{\overline{v}}}$ equipped with the inner product $\langle \Tilde{T_{\overline{v}}} (f) , \Tilde{T_{\overline{v}}} (g ) \rangle_{\ran} := \langle f, g \rangle_{H^2}$ is a RKHS with reproducing kernel at the point $x \in \mathbb{D}$ given by $\Tilde{T_{\overline{v}}}( v k_x)$.
\end{Proposition}

\begin{proof}
The inner product $\langle \, ,  \, \rangle_{\ran}$ is well defined because $\Tilde{T_{\overline{v}}}$ is injective. As $\langle \, ,  \, \rangle_{H^2}$ is an inner product, one can verify that $\langle \, ,  \, \rangle_{\ran}$ also satisfies the conditions to be an inner product. Because $\ran \Tilde{T_{\overline{v}}}$ is the isometric image of the complete space $H^2$, it also follows that $\ran \Tilde{T_{\overline{v}}}$ is complete. To see that $\Tilde{T_{\overline{v}}}( v k_x)$ is the reproducing kernel for $\ran \Tilde{T_{\overline{v}}}$ at $x \in \mathbb{D}$, observe that for all $f \in H^2$
$$
\langle \Tilde{T_{\overline{v}}}( f) , \Tilde{T_{\overline{v}}}( v k_x) \rangle_{\ran} = \langle f, v k_{x} \rangle_{H^{2}} = \int_{\mathbb{T}} \frac{f( \zeta) \overline{v}(\zeta)}{1-\overline{\zeta} x} \mathrm{d}m= P(\overline{v} f ) (x) = \Tilde{T_{\overline{v}}}( f) (x).
$$
\end{proof}
\begin{Remark}
Although the proposition above is in the context of Toeplitz operators, the same techniques show that for a bounded injective operator $X:  {H} \to B$, where $ {H}$ is a Hilbert space, one can equip $\ran X$ with an inner product which makes it a Hilbert space, even when $B$ itself is not a Hilbert space.
\end{Remark}

\begin{Proposition}
    Let $u, v \in H^2(\mathbb{T})$ be an admissible Toeplitz pair where $v$ is an outer function. Then $|u|^2 dm$ is a Carleson measure for $\ran \Tilde{T}_{\overline{v}}$ if and only if $uv \in L^{\infty}(\mathbb{T})$.
\end{Proposition}

\begin{proof}
As $\ran \Tilde{T}_{\overline{v}}$ is an isometric copy of $H^2(\mathbb{T})$, we see $|u|^2 dm$ is a Carleson measure for $\ran \Tilde{T}_{\overline{v}}$ if and only if $\int_{\mathbb{T}} |P({\overline{v}}f)|^2 |u|^2 \mathrm{d}m \leqslant  C^2 \|f\|_{\ran}^2 = C^2 \|f\|_{H^2}^2$. The result now follows from Theorem \ref{mainT}
\end{proof}

For convenience, we summarise the results of this section with a theorem.

\begin{Theorem}
    Let $u,v \in H^2(\mathbb{T})$ be an admissible Toeplitz pair with $v \in H^2(\mathbb{T})$ outer. The following are equivalent.
    \begin{enumerate}
        \item $T_u T_{\overline{v}}: H^2 \to H^2$ is bounded
        \item $ P_R: L^2\left(\mathbb{T}, \frac{1}{|v|^2}\right) \rightarrow L^2\left(\mathbb{T},|u|^2\right)$ is bounded
        \item $uv \in L^{\infty}(\mathbb{T})$
        \item $\sup _{|w|<1} P_w\left(|u|^2\right) P_w\left(|v|^2\right)<\infty$
        \item $|u|^2 dm$ is a Carleson measure for the space $\ran \Tilde T_{\overline{v}}$.
    \end{enumerate}
\end{Theorem}
\begin{proof}
    The commutative diagram \eqref{equiv} shows$(a)$ and $(b)$ are equivalent. Theorem \ref{mainT} shows $(a)$ and $(c)$ are equivalent. Corollary \ref{5.2} shows $(a)$ and $(d)$ are equivalent. The proposition above shows $(a)$ and $(e)$ are equivalent.
\end{proof}

\section*{Acknowledgements}
The author would like to thank Professor Jani Virtanen for the introduction to the problem and useful discussions.
\newline The author is grateful to EPSRC for financial support (grant EP/Y008375/1). 

\section*{Author information}
\author{Ryan O'Loughlin \\ E-mail address: \href{mailto:r.d.oloughlin@reading.ac.uk}{r.d.oloughlin@reading.ac.uk}}
\newline Department of Mathematics and Statistics, University of Reading, Reading, RG6 6AX, U.K.

 \newpage
 \bibliographystyle{plain}
\bibliography{bibliography.bib}

\begin{thebibliography}{10}

\bibitem{ahern1970radial}
P.R Ahern and D.N. Clark.
\newblock Radial limits and invariant subspaces.
\newblock {\em American Journal of Mathematics}, 92(2):332--342, 1970.

\bibitem{prodB3}
A.~Aleman, S.~Pott, and M.~C. Reguera.
\newblock {S}arason conjecture on the {B}ergman space.
\newblock {\em Int. Math. Res. Not. IMRN}, (14):4320--4349, 2017.

\bibitem{AlemanDuke}
A.~Aleman. and D.~Vukotić.
\newblock {Zero products of Toeplitz operators}.
\newblock {\em Duke Mathematical Journal}, 148(3):373 -- 403, 2009.

\bibitem{prodF3}
H.~Bommier-Hato, E.~H. Youssfi, and K.~Zhu.
\newblock {S}arason's {T}oeplitz product problem for a class of {F}ock spaces.
\newblock {\em Bull. Sci. Math.}, 141(5):408--442, 2017.

\bibitem{prodF2}
J.~J. Chen, X.~F. Wang, J.~Xia, and G.~F. Cao.
\newblock {S}arason's {T}oeplitz product problem on the {F}ock-{S}obolev space.
\newblock {\em Acta Math. Sin. (Engl. Ser.)}, 34(2):288--296, 2018.

\bibitem{cima2000backward}
J.~A. Cima and W.~T. Ross.
\newblock {\em The backward shift on the Hardy space}.
\newblock Number~79. American Mathematical Soc., 2000.

\bibitem{cohn1986radial}
W.S Cohn.
\newblock Radial limits and star invariant subspaces of bounded mean oscillation.
\newblock {\em American Journal of Mathematics}, 108(3):719--749, 1986.

\bibitem{prodH1}
D.~Cruz-Uribe.
\newblock The invertibility of the product of unbounded {T}oeplitz operators.
\newblock {\em Integral Equations Operator Theory}, 20(2):231--237, 1994.

\bibitem{rangespaces}
E.~Fricain, A.~Hartmann, and W.~T. Ross.
\newblock Range spaces of co-analytic {T}oeplitz operators.
\newblock {\em Canad. J. Math.}, 70(6):1261--1283, 2018.

\bibitem{fricain2019multipliers}
E.~Fricain, A.~Hartmann, and W.~T. Ross.
\newblock Multipliers between range spaces of co-analytic {T}oeplitz operators.
\newblock {\em Acta Scientiarum Mathematicarum}, 85(1):215--230, 2019.

\bibitem{sarasonprobbook}
V.~P. Havin and N.~K. Nikolski, editors.
\newblock {\em Linear and complex analysis. {P}roblem book 3. {P}art {I}}, volume 1573 of {\em Lecture Notes in Mathematics}.
\newblock Springer-Verlag, Berlin, 1994.

\bibitem{HarmA1}
B.~Hunt, R.and~Muckenhoupt and R.~Wheeden.
\newblock Weighted norm inequalities for the conjugate function and {Hil}bert transform.
\newblock {\em Transactions of the American Mathematical Society}, 176:227--251, 1973.

\bibitem{HarmA5}
M.~T Lacey.
\newblock The two weight inequality for the {Hil}bert transform: a primer.
\newblock In {\em Harmonic Analysis, Partial Differential Equations, Banach Spaces, and Operator Theory (Volume 2) Celebrating Cora Sadosky's Life}, pages 11--84. Springer, 2017.

\bibitem{HarmA2}
M.~T. Lacey, Eric~T. Sawyer, Chun-Yen Shen, and Ignacio Uriarte-Tuero.
\newblock {Two-weight inequality for the {Hil}bert transform: A real variable characterization, I}.
\newblock {\em Duke Mathematical Journal}, 163(15):2795 -- 2820, 2014.

\bibitem{HarmA3}
M.T. Lacey.
\newblock {Two-weight inequality for the {Hil}bert transform: A real variable characterization, II}.
\newblock {\em Duke Mathematical Journal}, 163(15):2821 -- 2840, 2014.

\bibitem{mccarthy1990common}
J.~E McCarthy.
\newblock Common range of co-analytic {T}oeplitz operators.
\newblock {\em Journal of the American Mathematical Society}, 3(4):793--799, 1990.

\bibitem{prodH3}
F.~Nazarov.
\newblock A counter-example to {S}arason’s conjecture.
\newblock {\em preprint; available at http://www.math.msu.edu/fedja/prepr.html}.

\bibitem{toeplitzkernelbackwardarticle}
R.~O'Loughlin.
\newblock {T}oeplitz kernels and the backward shift.
\newblock {\em J. Math. Anal. Appl.}, 492(2):124489, 22, 2020.

\bibitem{o2021multidimensional}
R.~O'Loughlin.
\newblock Multidimensional {T}oeplitz and truncated {T}oeplitz operators.
\newblock {\em PhD thesis; available at https://etheses.whiterose.ac.uk/id/eprint/29284/}, 2021.

\bibitem{paulsen2016introduction}
V.~I Paulsen and M.~Raghupathi.
\newblock {\em An introduction to the theory of reproducing kernel Hilbert spaces}, volume 152.
\newblock Cambridge university press, 2016.

\bibitem{sarason1994sub}
D.~Sarason.
\newblock Sub-{H}ardy {H}ilbert spaces in the unit disk.
\newblock {\em University of Arkansas Lecture Notes in the Mathematical Sciences, 10. John Wiley and Sons Inc., New York, 1994. A Wiley-Interscience Publication}, 1994.

\bibitem{HarmA4}
E.T Sawyer and B.~D Wick.
\newblock The {H}yt{\"o}nen-{V}uorinen ${L}^{p}$ conjecture for the {H}ilbert transform when $(4/3)< p< 4$ and the measures share no point masses.
\newblock {\em arXiv preprint arXiv:2308.10733}, 2023.

\bibitem{prodB2}
K.~Stroethoff and D.~Zheng.
\newblock Products of {H}ankel and {T}oeplitz operators on the {B}ergman space.
\newblock {\em J. Funct. Anal.}, 169(1):289--313, 1999.

\bibitem{stroethoff2002invertible}
K.~Stroethoff and D.~Zheng.
\newblock Invertible {T}oeplitz products.
\newblock {\em Journal of Functional Analysis}, 195(1):48--70, 2002.

\bibitem{prodB1}
K.~Stroethoff and D.~Zheng.
\newblock Bounded {T}oeplitz products on the {B}ergman space of the polydisk.
\newblock {\em J. Math. Anal. Appl.}, 278(1):125--135, 2003.

\bibitem{Toeplitzsoriginal1}
O.~Toeplitz.
\newblock Zur theorie der quadratischen und bilinearen {F}ormen von unendlichvielen {V}eränderlichen, i.
\newblock {\em Theorie des L-Formen. Math. Ann}, 70:351--376, 1911.

\bibitem{Toeplitzsoriginal2}
O.~Toeplitz.
\newblock Über die {F}ouriersche {E}ntwickelung positiver {F}unktionen.
\newblock {\em Rend.Circ. Mat. Palermo}, 32:191--192, 1911.

\bibitem{prodH2}
D.~Zheng.
\newblock The distribution function inequality and products of {T}oeplitz operators and {H}ankel operators.
\newblock {\em J. Funct. Anal.}, 138(2):477--501, 1996.

\end{thebibliography}
\end{document}